\documentclass[12pt]{amsart}
\usepackage{amssymb}
\usepackage{amsmath, amscd}
\usepackage{amsthm}
\usepackage{xcolor}

\newtheorem{theorem}{Theorem}[section]

\newtheorem{proposition}[theorem]{Proposition}
\newtheorem{lemma}[theorem]{Lemma}

\newtheorem{corollary}[theorem]{Corollary}
\theoremstyle{definition}

\newtheorem{example}[theorem]{Example}
\newtheorem{definition}[theorem]{Definition}

\newtheorem{remark}[theorem]{Remark}

\newtheorem{conjecture}[theorem]{Conjecture}

\newcommand{\qihao}{\fontsize{5.0pt}{\baselineskip}\selectfont}

\begin{document}

\author[J. Cui]{Jian Cui}
\address{School of Mathematics and Statistics, Anhui Normal University \\ Wuhu, Anhui 241002, China}
\email{cui368@ahnu.edu.cn}
\author[P. Danchev]{Peter Danchev}
\address{Institute of Mathematics and Informatics, Bulgarian Academy of Sciences \\ "Acad. G. Bonchev" str., bl. 8, 1113 Sofia, Bulgaria}
\email{danchev@math.bas.bg; pvdanchev@yahoo.com} \vspace{0.5cm}
\author[Y. Zeng]{Yuedi Zeng}
\address{Key Laboratory of
Financial Mathematics (Putian University), Fujian Province
University, Putian, Fujian, China} 
\email{155221109@qq.com}

\title[Some New Results on Pseudo $n$-Strong Drazin Inverses ...] {Some New Results on Pseudo $n$-Strong \\ Drazin Inverses in Rings}

\maketitle

\begin{abstract} In this paper, we
give a further study in-depth of the pseudo $n$-strong Drazin
inverses in an associative unital ring $R$. The characterizations of
elements $a,b\in R$ for which $aa^{\tiny{\textcircled{\qihao
D}}}=bb^{\tiny{\textcircled{\qihao D}}}$ are provided, and some new
equivalent conditions on pseudo $n$-strong Drazin inverses are
obtained. In particular, we show that an element $a\in R$ is pseudo
$n$-strong Drazin invertible if, and only if, $a$ is $p$-Drazin
invertible and $a-a^{n+1}\in \sqrt{J(R)}$ if, and only if, there
exists $e^2=e\in {\rm comm}^2(a)$ such that $ae\in \sqrt{J(R)}$ and
$1-(a+e)^n\in \sqrt{J(R)}$. We also consider pseudo $n$-strong
Drazin inverses with involution, and discuss the extended versions
of Cline's formula and Jacobson's lemma of this new class of
generalized inverses. Likewise, we define and explore the so-called
{\it pseudo $\pi$-polar} rings and demonstrate their relationships
with periodic rings and strongly $\pi$-regular rings, respectively.
\end{abstract}

\bigskip

\section{\bf Introduction and Motivation}

Let $R$ be an arbitrary associative ring with unity. The symbols $U(R)$, $J(R)$ and $R^{\rm nil}$ stand,  respectively, for the set of all invertible elements, the Jacobson radical and the set of all nilpotent elements of $R$. We use ${\rm M}_n(R)$ to denote the $n\times n$ matrix ring over $R$. For any $a\in R$, the commutant and double commutant of $a$ are defined by ${\rm comm}(a)=\{x\in R:ax=xa\}$ and ${\rm comm}^{2}(a)$$=\{x\in R:xy=yx,{\rm ~for~all}~y\in {\rm comm}(a)\}$, respectively.

Recall that an element $a\in R$ is said to be \emph{Drazin
invertible} if there exists $x\in R$ such that $$xax=x,~  ~x\in {\rm
comm}^2(a),~ ~a-a^{2}x\in R^{\rm nil}. \eqno{(\ast})$$ Herein, the
element $x$ is called a \emph{Drazin inverse} of $a$, which is
unique if it exists. It is well known that the condition ``$x\in
{\rm comm}^2(a)$" in $(\ast)$ can be replaced with ``$x\in {\rm
comm}(a)$". The notion of Drazin inverses was initially introduced
by Drazin \cite{Drazin} in 1958, which plays important roles in
matrix theory, Markov chains, partial differential equations, etc.
(see \cite{ben,Cam,Na}). Moreover, Drazin inverses are closely
related to the well-known strongly $\pi$-regular rings in ring
theory. In fact, a ring $R$ is \emph{strongly $\pi$-regular} (see
\cite{Tu}) if, for each $a\in R$, the relation $a^n\in a^{n+1}R\cap
Ra^{n+1}$ holds for some positive integer $n$ depending on $a$
(which ring is also called \emph{polar} in the existing literature).
It is well known that a ring $R$ is strongly $\pi$-regular precisely
when every element of $R$ is Drazin invertible.

For a ring $R,$ let us denote $\sqrt{J(R)}=\{a\in R:a^{k}\in
J(R)~{\rm for~some}~k\geq1\}.$ It is easy to see that both $R^{\rm
nil}$ and $J(R)$ are contained in $\sqrt{J(R)}$.

In 2012, Wang and Chen \cite{{Wang12}} introduced the notion of the pseudo Drazin inverse in rings. Indeed, an element $x\in R$ is called \emph{pseudo Drazin invertible} (or, for short, \emph{$p$-Drazin invertible}) if there exists $x\in R$ such that
\begin{center}
$xax=x$,~~$x\in {\rm comm}^{2}(a)$,~~$a-a^{2}x\in \sqrt{J(R)}$.~
\end{center}
Any element $x\in R$ satisfying the above conditions is also called a \emph{pseudo Drazin inverse} (or, shortly, \emph{$p$-Drazin inverse}) of $a$; notice that the pseudo Drazin inverse is unique once it exists. Clearly, Drazin
invertible elements are $p$-Drazin invertible. According to \cite{Wang12}, $a\in R$ is $p$-Drazin invertible exactly when $a$ is a pseudo-polar element (i.e., there exists $p^2=p\in {\rm comm}^2(a)$ such that $a+p\in U(R)$ and $ap\in \sqrt{J(R)}$). For more interesting properties of $p$-Drazin inverses, one may refer to \cite{Drazin14,Mosic20,Zhu,Zou1} and the reference therewith.

In 2021, Mosi$\acute{c}$ \cite{Mosic21} introduced the concept of
the pseudo $n$-strong Drazin inverse which generalizes the concept
of $p$-Drazin inverses. Let $\mathbb{N}$ be the set of all positive
integers and $n\in \mathbb{N}$. An element $a\in R$ is called
\emph{pseudo $n$-strong Drazin invertible} (or, briefly,
\emph{$pns$-Drazin invertible}) if there exists $x\in R$ such that
\begin{center}
$xax=x$,~~$x\in {\rm comm}^{2}(a)$,~~$a^{n}-ax\in \sqrt{J(R)},$
\end{center}
where $x$ is a \emph{pseudo $n$-strong Drazin inverse} (or, for short, \emph{$pns$-Drazin inverse}) of $a$, which is unique when it exists, and denoted by $a^{\tiny{\textcircled{\qihao D}}}$. It was shown that any $pns$-Drazin invertible element is $p$-Drazin invertible (see, for instance, \cite{Mosic21}).


Motivated by all of the mentioned above information, we continue in
the present paper the investigation of $pns$-Drazin inverses in
arbitrary (associative) rings. To that goal, our work is organized
thus: in Section 2, the elements $a,b\in R$ for which
$aa^{\tiny{\textcircled{\qihao D}}}=bb^{\tiny{\textcircled{\qihao
D}}}$ are studied, and some equivalent characterizations are
obtained. In particular, it is established that $a\in R$ is
$pns$-Drazin invertible if, and only if, $a$ is $p$-Drazin
invertible and $a-a^{n+1}\in \sqrt{J(R)}$ if, and only if, there
exists $p^2=p\in {\rm comm}^2(a)$ such that $ap\in \sqrt{J(R)}$ and
$1-(a+p)^n\in \sqrt{J(R)}.$ Further, Section 3 is devoted to the
exploration of $pns$-Drazin inverses with involution. Recall that a
ring $R$ is a \emph{$*$-ring} (see, e.g., \cite{Ber72}) if there
exists a map $*:R\rightarrow R$ such that $$(x+y)^*=x^*+y^*, ~
(xy)^*=y^*x^*, ~ (x^*)^*=x$$ for all $x,y\in R$. For a $*$-ring $R$
and $n\in \mathbb{N}$, we shall say that the element $a\in R$ is
{\it $pns$-$*$-Drazin invertible} if $a$ is $pns$-Drazin invertible
and $(aa^{\tiny{\textcircled{\qihao
D}}})^*=aa^{\tiny{\textcircled{\qihao D}}}$. On this vein, some
properties of $pns$-$*$-Drazin inverses are discussed too. It is
also proved that, for a $*$-ring $R$ and elements $a$, $b$, $c\in
R$, if $aba=ba^{2}=a^{2}c=aca$, then $ac$ (resp., $1-ac$) is
$pns$-$*$-Drazin invertible if, and only if, so is $ba$ (resp.,
$1-ba$) and, moreover, the related formulae are provided. In our
final Section 4, we deal with some characterizing theorems
pertaining to pseudo $\pi$-polar rings, that are those rings whose
elements are $pns$-Drazin invertible for various positive integers
$n$, examining their basic properties and transversal with the
classical sorts of periodic and strongly $\pi$-regular rings by
showing that any pseudo $\pi$-polar ring $R$ with nil $J(R)$ is
periodic and vice versa. Some other closely related assertions are
proved as well.

\medskip


\medskip

\section{\bf $pns$-Drazin inverses}

The following easy lemma will be used freely in the context.

\begin{lemma}\label{1-2}
Let $a,b\in R$ and $ab=ba$.
\par$(1)$ If $a\in \sqrt{J(R)}$, then $ab\in \sqrt{J(R)}.$
\par$(2)$ If $a,b\in \sqrt{J(R)}$, then $a+b\in\sqrt{J(R)}$.
\par$(3)$ If $a\in \sqrt{J(R)}, b\in U(R)$, then $a+b\in U(R)$.
\par$(4)$ $a\in \sqrt{J(R)}$ if, and only if, $au\in \sqrt{J(R)}$ for any $u\in U(R)\cap {\rm
comm}(a)$ if, and only if, $a^k\in\sqrt{J(R)}$ for any integer $k>1$.
\end{lemma}

The next lemma appeared as \cite[Theorem 3.1]{Mosic21}.

\begin{lemma}\label{1-1}
Let $a\in R$ and $n\in \mathbb{N}$. The following are equivalent$:$
\par$(1)$ The element $a$ is $pns$-Drazin invertible.
\par$(2)$ There exists $e^2=e\in {\rm comm}^2{(a)}$ such that $a^n-e\in
\sqrt{J(R)}.$\\
In this case, $e=aa^{\tiny{\textcircled{\qihao D}}}$ and
$a^{\tiny{\textcircled{\qihao D}}}=(1+a^n-e)^{-1}a^{n-1}e.$
\end{lemma}

For the sake of convenience, we use the symbol $R^{\tiny{\textcircled{\qihao D}}}$ to denote the set of all $pns$-Drazin invertible elements in a ring $R$. If $a\in R^{\tiny{\textcircled{\qihao D}}}$, we write
$a^{\Pi}=aa^{\tiny{\textcircled{\qihao D}}}$. Then, $a^{\Pi}=(a^{\Pi})^2$, $a^{\tiny{\textcircled{\qihao
D}}}=(1+a^n-a^{\Pi})^{-1}a^{n-1}a^{\Pi}$ and $a^{\tiny{\textcircled{\qihao D}}}
a^{\Pi}=a^{\Pi}a^{\tiny{\textcircled{\qihao D}}}=a^{\tiny{\textcircled{\qihao D}}}$.

As usual, for an element $a\in R$, we denote by $r(a)$ the right
annihilator of $a$ in $R$. Our motivating of the following result is
originated from \cite[Theorem 6.1]{Koliha}.

\begin{theorem}\label{1234}
Let $a\in R^{\tiny{\textcircled{\qihao D}}}$, $b\in R$ and $n\in
\mathbb{N}$. The following are equivalent$:$
\par$(1)$ $b\in R^{\tiny{\textcircled{\qihao D}}}$ and $a^{\Pi}=b^{\Pi}$.
\par$(2)$ $a^{\Pi}\in {\rm comm}^2(b)$ and $b^n-a^{\Pi}\in \sqrt{J(R)}.$
\par$(3)$ $a^{\Pi}\in {\rm comm}^2(b)$, $b(1-a^{\Pi})\in
\sqrt{J(R)}$ and $(1-b^n)a^{\Pi}\in \sqrt{J(R)}$.
\par$(4)$ $b\in R^{\tiny{\textcircled{\qihao D}}},$ $1-a^{\Pi}(1-a^{\tiny{\textcircled{\qihao D}}}
b)\in U(R)$ and $b^{\tiny{\textcircled{\qihao
D}}}=(1-a^{\Pi}(1-a^{\tiny{\textcircled{\qihao D}}}
b))^{-1}a^{\tiny{\textcircled{\qihao D}}}$.
\par$(5)$ $b\in R^{\tiny{\textcircled{\qihao D}}}$ and $b^{\tiny{\textcircled{\qihao D}}}-a^{\tiny{\textcircled{\qihao D}}}=a^{\tiny{\textcircled{\qihao D}}}(a-b)b^{\tiny{\textcircled{\qihao D}}}.$
\par$(6)$ $b\in R^{\tiny{\textcircled{\qihao D}}}$, $a^{\Pi}b^{\Pi}=b^{\Pi}a^{\Pi}$ and $1-(a^{\Pi}-b^{\Pi})^2\in
U(R)$.
\par$(7)$ $b\in R^{\tiny{\textcircled{\qihao D}}}$, $b^{\tiny{\textcircled{\qihao D}}} R\subseteq a^{\tiny{\textcircled{\qihao D}}} R$
and $r(b^{\tiny{\textcircled{\qihao D}}})\subseteq
r(a^{\tiny{\textcircled{\qihao D}}})$.
\end{theorem}

\begin{proof}
$(1)\Leftrightarrow(2)$. It follows immediately from Lemma \ref{1-1}.

$(2)\Leftrightarrow(3)$. Note that $ba^{\Pi}=a^{\Pi}b$ is implied by $a^{\Pi}\in {\rm comm}^2(b)$. Therefore,
$$b^n-a^{\Pi}\in \sqrt{J(R)}\Leftrightarrow (b^n-a^{\Pi})(1-a^{\Pi})\in \sqrt{J(R)},~(b^n-a^{\Pi})a^{\Pi}\in
\sqrt{J(R)},$$ and $$(b^n-a^{\Pi})(1-a^{\Pi})=b^n(1-a^{\Pi})\in \sqrt{J(R)}\Leftrightarrow  b(1-a^{\Pi})\in \sqrt{J(R)}.$$
Furthermore, $(b^n-a^{\Pi})a^{\Pi}=(b^n-1)a^{\Pi} \in \sqrt{J(R)}
\Leftrightarrow (1-b^n)a^{\Pi}\in \sqrt{J(R)}.$ This proves the
equivalence between $(2)$ and $(3)$.

$(1)\Rightarrow(4)$. By $(1)$, $a^{\Pi}=b^{\Pi}$. Observe that
$1-a^{\Pi}(1-a^{\tiny{\textcircled{\qihao D}}}
b)=(1-a^{\Pi})+a^{\tiny{\textcircled{\qihao D}}} b$. Consequently,
\begin{equation*}
\begin{split}
[(1-a^{\Pi})+a^{\tiny{\textcircled{\qihao D}}}
b][(1-a^{\Pi})+b^{\tiny{\textcircled{\qihao D}}} a ]
&=(1-a^{\Pi})+a^{\tiny{\textcircled{\qihao D}}} b(1-a^{\Pi})+(1-a^{\Pi})b^{\tiny{\textcircled{\qihao D}}} a+a^{\tiny{\textcircled{\qihao D}}} b b^{\tiny{\textcircled{\qihao D}}} a\\
&=1-a^{\Pi}+a^{\tiny{\textcircled{\qihao D}}} (1-a^{\Pi}) b+(1-b^{\Pi})b^{\tiny{\textcircled{\qihao D}}} a+a^{\tiny{\textcircled{\qihao D}}}  b^{\Pi} a\\
&=1-a^{\Pi}+a^{\tiny{\textcircled{\qihao D}}}  a^{\Pi} a=1.
\end{split}
\end{equation*}
We can similarly deduce that
$[(1-a^{\Pi})+b^{\tiny{\textcircled{\qihao D}}} a
][(1-a^{\Pi})+a^{\tiny{\textcircled{\qihao D}}} b]=1$. So
$1-a^{\Pi}(1-a^{\tiny{\textcircled{\qihao D}}}
b)=(1-a^{\Pi})+a^{\tiny{\textcircled{\qihao D}}} b\in U(R)$. It
follows that
$$(1-a^{\Pi}(1-a^{\tiny{\textcircled{\qihao D}}} b))b^{\tiny{\textcircled{\qihao D}}}=(1-a^{\Pi}+a^{\tiny{\textcircled{\qihao D}}}
b)b^{\tiny{\textcircled{\qihao
D}}}=(1-b^{\Pi})b^{\tiny{\textcircled{\qihao D}}}
+a^{\tiny{\textcircled{\qihao D}}} bb^{\tiny{\textcircled{\qihao
D}}}=a^{\tiny{\textcircled{\qihao D}}}
b^{\Pi}=a^{\tiny{\textcircled{\qihao D}}}
a^{\Pi}=a^{\tiny{\textcircled{\qihao D}}},$$ which implies that
$b^{\tiny{\textcircled{\qihao
D}}}=(1-a^{\Pi}(1-a^{\tiny{\textcircled{\qihao D}}}
b))^{-1}a^{\tiny{\textcircled{\qihao D}}}$.

$(4)\Rightarrow(5)$. By $(4),$ we have $a^{\tiny{\textcircled{\qihao
D}}} =(1-a^{\Pi}(1-a^{\tiny{\textcircled{\qihao D}}}
b))b^{\tiny{\textcircled{\qihao D}}}$. Thus,
\begin{center}
$b^{\tiny{\textcircled{\qihao D}}}-a^{\tiny{\textcircled{\qihao
D}}}=b^{\tiny{\textcircled{\qihao
D}}}-(1-a^{\Pi}(1-a^{\tiny{\textcircled{\qihao D}}}
b))b^{\tiny{\textcircled{\qihao D}}}
=(a^{\Pi}-a^{\tiny{\textcircled{\qihao D}}}
b)b^{\tiny{\textcircled{\qihao D}}} =a^{\tiny{\textcircled{\qihao
D}}}(a-b)b^{\tiny{\textcircled{\qihao D}}}.$
\end{center}

$(5)\Rightarrow(1)$. Since $b^{\tiny{\textcircled{\qihao
D}}}-a^{\tiny{\textcircled{\qihao D}}}=a^{\tiny{\textcircled{\qihao
D}}}(a-b)b^{\tiny{\textcircled{\qihao D}}},$ one has
$b^{\tiny{\textcircled{\qihao D}}}=a^{\tiny{\textcircled{\qihao
D}}}[1+(a-b)b^{\tiny{\textcircled{\qihao D}}}]$ and
$a^{\tiny{\textcircled{\qihao D}}}=[1-a^{\tiny{\textcircled{\qihao
D}}}(a-b)]b^{\tiny{\textcircled{\qihao D}}}$. Then,
$(1-a^{\Pi})b^{\tiny{\textcircled{\qihao
D}}}=(1-a^{\Pi})a^{\tiny{\textcircled{\qihao
D}}}[1+(a-b)b^{\tiny{\textcircled{\qihao D}}}]=0$ and
$a^{\tiny{\textcircled{\qihao
D}}}(1-b^{\Pi})=[1-a^{\tiny{\textcircled{\qihao
D}}}(a-b)]b^{\tiny{\textcircled{\qihao D}}}(1-b^{\Pi})=0$, it
follows that $0=(1-a^{\Pi})b^{\tiny{\textcircled{\qihao D}}}
b=(1-a^{\Pi})b^{\Pi}$ and $0=aa^{\tiny{\textcircled{\qihao
D}}}(1-b^{\Pi})=a^{\Pi}(1-b^{\Pi})$. Therefore,
$a^{\Pi}=a^{\Pi}b^{\Pi}=b^{\Pi}$.

$(1)\Rightarrow(6)$. It is clear.

$(6)\Rightarrow(1)$. From $a^{\Pi}b^{\Pi}=b^{\Pi}a^{\Pi}$, we have
$(a^{\Pi}-b^{\Pi})=(a^{\Pi}-b^{\Pi})^3$. So $1-(a^{\Pi}-b^{\Pi})^2$
is an idempotent. By hypothesis, $1-(a^{\Pi}-b^{\Pi})^2\in U(R).$
Thus, $1-(a^{\Pi}-b^{\Pi})^2=1,$ which implies
$(a^{\Pi}-b^{\Pi})^2=0$, and hence
$a^{\Pi}-b^{\Pi}=(a^{\Pi}-b^{\Pi})^3=0.$ So, the result follows.

$(1)\Rightarrow(7)$. Since $a^{\Pi}=b^{\Pi},$ we obtain
$b^{\tiny{\textcircled{\qihao D}}}
R=(b^{\Pi}b^{\tiny{\textcircled{\qihao D}}})
R=(a^{\Pi}b^{\tiny{\textcircled{\qihao D}}}) R
=(a^{\tiny{\textcircled{\qihao D}}} a b^{\tiny{\textcircled{\qihao
D}}}) R \subseteq a^{\tiny{\textcircled{\qihao D}}} R$. Similarly, $
Ra^{\tiny{\textcircled{\qihao D}}}=R(a^{\tiny{\textcircled{\qihao
D}}} a^{\Pi}) =R(a^{\tiny{\textcircled{\qihao D}}}
b^{\Pi})=R(a^{\tiny{\textcircled{\qihao D}}}
b^{\Pi}b^{\tiny{\textcircled{\qihao D}}})\subseteq
Rb^{\tiny{\textcircled{\qihao D}}}$, which forces that
$r(b^{\tiny{\textcircled{\qihao
D}}})=r(Rb^{\tiny{\textcircled{\qihao D}}})\subseteq
r(Ra^{\tiny{\textcircled{\qihao D}}})=r(a^{\tiny{\textcircled{\qihao
D}}})$.

$(7)\Rightarrow(1)$. One observes that $b^{\tiny{\textcircled{\qihao
D}}}=b^{\tiny{\textcircled{\qihao D}}} b
b^{\tiny{\textcircled{\qihao D}}}$ and $b^{\tiny{\textcircled{\qihao
D}}} b=b^{\Pi}$. Then, $b^{\tiny{\textcircled{\qihao D}}} R =
b^{\tiny{\textcircled{\qihao D}}} b R=b^{\Pi}R$ and
$Rb^{\tiny{\textcircled{\qihao D}}}=Rbb^{\tiny{\textcircled{\qihao
D}}}=Rb^{\Pi}.$ We can prove similarly that $
a^{\tiny{\textcircled{\qihao D}}} R =a^{\tiny{\textcircled{\qihao
D}}} a R=a^{\Pi}R$ and $Ra^{\tiny{\textcircled{\qihao
D}}}=Raa^{\tiny{\textcircled{\qihao D}}}=Ra^{\Pi}.$ Since
$b^{\tiny{\textcircled{\qihao D}}} R\subseteq
a^{\tiny{\textcircled{\qihao D}}} R$, we get $b^{\Pi}R\subseteq
a^{\Pi}R$, which gives $b^{\Pi}=a^{\Pi}b^{\Pi}$. Note that
$r(b^{\tiny{\textcircled{\qihao
D}}})=r(Rb^{\tiny{\textcircled{\qihao
D}}})=r(Rb^{\Pi})=(1-b^{\Pi})R$ and $r(a^{\tiny{\textcircled{\qihao
D}}})=r(Ra^{\tiny{\textcircled{\qihao
D}}})=r(Ra^{\Pi})=(1-a^{\Pi})R$. Since
$r(b^{\tiny{\textcircled{\qihao D}}})\subseteq
r(a^{\tiny{\textcircled{\qihao D}}})$, we have
$(1-b^{\Pi})R\subseteq (1-a^{\Pi})R,$ which yields
$1-b^{\Pi}=(1-a^{\Pi})(1-b^{\Pi}).$ Thus, $a^{\Pi}=a^{\Pi}b^{\Pi}$
and, therefore, $a^{\Pi}=b^{\Pi}$.
\end{proof}

As a consequence, we derive the following.

\begin{corollary}\label{1111}
Let $a\in R$, $e^2=e\in R$ and $n\in \mathbb{N}$. The following are equivalent$:$
\par$(1)$ $a$ is $pns$-Drazin invertible with $a^{\Pi}=e$.
\par$(2)$ $e^2=e\in {\rm comm}^2{(a)}$, $a(1-e)\in \sqrt{J(R)}$ and $(1-a^n)e\in \sqrt{J(R)}.$
\end{corollary}

\begin{proof}
Given any $f^2=f\in R,$ we can check that $f\in
R^{\tiny{\textcircled{\qihao D}}}$ and $f^{\tiny{\textcircled{\qihao
D}}}=f$, whence $f^{\Pi}=ff^{\tiny{\textcircled{\qihao D}}}=f$.

$(1)\Rightarrow(2)$. Since $e^2=e\in R$, one sees that $e\in
R^{\tiny{\textcircled{\qihao D}}}$ with $e^{\Pi}=e$. So, we have
$a^{\Pi}=e^{\Pi}$. In view of Theorem \ref{1234}, $e^2=e\in {\rm
comm}^2{(a)}$, $a(1-e)\in \sqrt{J(R)}$ and $(1-a^n)e\in
\sqrt{J(R)}$.

$(2)\Rightarrow(1)$. As $e=e^{\Pi}$, it follows at once from Theorem
\ref{1234} that $a\in R^{\tiny{\textcircled{\qihao D}}}$ and
$a^{\Pi}=e^{\Pi}=e$.
\end{proof}

By virtue of Corollary \ref{1111}, we extract the following assertion.

\begin{corollary}
Let $a\in R$ and $n\in \mathbb{N}$. Then,
\par$(1)$ $a$ is $pns$-Drazin
invertible with $a^{\Pi}=1$ if, and only if, $a^n-1\in \sqrt{J(R)}.$
\par$(2)$ $a$ is $pns$-Drazin
invertible with $a^{\Pi}=0$ if, and only if, $a\in \sqrt{J(R)}.$
\end{corollary}



Recall again that an element $a\in R$ is \emph{$p$-Drazin invertible} if there exists $y\in R$ such that
$$yax=y, ~~y\in {\rm comm}^{2}(a)~~ {\rm and}~ a-a^{2}y\in \sqrt{J(R)},$$ where $y$ is called a $p$-Drazin inverse of $a$.

\medskip

The following result establishes the relationship between a $pns$-Drazin invertible element and a $p$-Drazin invertible element in a ring.

\begin{theorem}\label{1-3}
Let $a\in R$ and $n\in \mathbb{N}$. The following are equivalent$:$
\par$(1)$ The element $a$ is $pns$-Drazin invertible.
\par$(2)$ There exists a unique $e^2=e\in {\rm comm}{(a)}$ such that $a^n-e\in \sqrt{J(R)}.$
\par$(3)$ There exists a unique $y\in R$ such that $ya=ay$, $y=yay$ and $a^n-ay\in \sqrt{J(R)}.$
\par$(4)$ $a$ is $p$-Drazin invertible and $a-a^{n+1}\in
\sqrt{J(R)}$.
\end{theorem}

\begin{proof}
$(1)\Rightarrow(4)$. Clearly, $pns$-Drazin invertible elements are
$p$-Drazin invertible. In view of Lemma \ref{1-1}, there exists
$e^2=e\in {\rm comm}^2{(a)}$ such that $a^n-e\in \sqrt{J(R)}$. As
$ae=ea,$ we have $$a^{n+1}-ae=a(a^n-e)\in \sqrt{J(R)},$$ and
$$[a(1-e)]^n=(a^n-e)(1-e)\in \sqrt{J(R)},$$ which implies $a(1-e)\in
\sqrt{J(R)}.$ It now follows that $$a-a^{n+1}=a(1-e)-(a^{n+1}-ae)\in
\sqrt{J(R)},$$ as wanted.

$(4)\Rightarrow(2)$. Since $a$ is $p$-Drazin invertible, with
\cite[Theorem 3.2]{Wang12} at hand there exists $p^2=p\in {\rm comm}^2(a)$
such that $a+p\in U(R)$ and $ap\in \sqrt{J(R)}$. Write $u=a+p$ and
$e=1-p$. Then, $ep=pe=0$ and $ae=ue=eu=ea$. Since $a-a^{n+1}\in
\sqrt{J(R)}$, it must be that $$(a-a^{n+1})e=(u-u^{n+1})e\in \sqrt{J(R)},$$ which
implies that $(1-u^n)e\in \sqrt{J(R)},$ and so
$$(a^n-e)e=u^ne-e=(u^n-1)e\in \sqrt{J(R)}.$$ As $ap\in \sqrt{J(R)}$, we get
$(a^n-e)(1-e)=(ap)^n\in \sqrt{J(R)}.$ Therefore,
$$a^n-e=(a^n-e)e+(a^n-e)(1-e)\in \sqrt{J(R)}.$$ To show that the
idempotent $e$ above is unique, we assume that there exists
$f^2=f\in {\rm comm}(a)$ satisfying $a^n-f\in \sqrt{J(R)}.$ Set
$a^n-f=b.$ Then, $a^nf=(f+b)f=(2f-1+b)f$, where $2f-1+b\in R^{-1}$ as
$(2f-1)^2=1$ and $fb=bf$. Thus, $f=(2f-1+b)^{-1}a^nf$. From $e\in
{\rm comm}^2(a)$, one has $(1-e)f$ is an idempotent. However,
$$(1-e)f=(1-e)[(2f-1+b)^{-1}a^nf]=(a^n-e)(1-e)(2f-1+b)^{-1}f\in
\sqrt{J(R)}$$ since $a^n-e\in \sqrt{J(R)}$ and $e,f,a,b$ commute with
one another. This proves that $(1-e)f=0$. Similarly, we can show
that $(1-f)e=0$. Hence, $e=fe=ef=f,$ as desired.

$(2)\Rightarrow(3)$. Assume that $(2)$ holds. Put
$y=(1+a^n-e)^{-1}a^{n-1}e$. Then, one inspects that
$$ya=(1+a^n-e)^{-1}a^{n}e=(1+a^n-e)^{-1}(a^{n}+1-e)e=e=ay,$$
$yay=ye=y$ and $a^n-ay=a^n-e\in \sqrt{J(R)}$. Suppose now that there
exists another $z\in R$ such that $za=az$, $z=zaz$ and $a^n-az\in
\sqrt{J(R)}$. Let $f=az.$ Then, $f\in {\rm comm}(a)$ and
$a^n-f=a^n-az\in \sqrt{J(R)}$. By hypothesis, $e=f.$ So,
$az=za=ya=ay,$ which assures $y=yay=zay=zaz=z$, as expected.

$(3)\Rightarrow(2)$. This implication follows in a similar manner to
that of $(2)\Rightarrow(3)$.

$(2)\Rightarrow(1)$. Using Lemma \ref{1-1}, it suffices to show that
$e\in {\rm comm}^2(a)$. For any $x\in {\rm comm}(a),$ set
$p=e+ex(1-e)$. Evidently, $p^2=p$ and $p\in {\rm comm}(a)$. Since
$a^n-e=:c\in \sqrt{J(R)},$ we obtain $a^n=(1-e)+(2e-1+c)$ and
$(a^n-e)^k\in J(R)$ for some integer $k$. It, therefore, follows that
$a^ne=(2e-1+c)e$ and $a^{nk}(1-e)=(a^n-e)^k(1-e)\in J(R).$ Write
$v=2e-1+c$. Then, $v\in U(R)$ and $e=v^{-1}a^ne.$ Since $ax=xa$ and
$a^{nk}(1-e)\in J(R),$ we conclude that
$$ex(1-e)=e^kx(1-e)=v^{-k}a^{nk}ex(1-e)=v^{-k}exa^{nk}(1-e)\in J(R),$$
which insures that
$$a^n-p=(a^n-e)-ex(1-e)\in \sqrt{J(R)}.$$ By assumption, $e=p$.
Thus, we obtain $ex=exe.$ Similarly, we can obtain $xe=exe$. Therefore,
$ex=xe$, and so $e\in {\rm comm}^2(a)$, as required.
\end{proof}

By the usage of \cite[Theorem 3.1]{Mosic21}, \cite[Theorem
3.2]{Wang12} and the proof of Theorem \ref{1-3}, we provide a
transparent connection between a $p$-Drazin inverse and a
$pns$-Drazin inverse in a ring as follows.

\begin{corollary}\label{1-4}
Let $a,y\in R$, $e=ay$ and $n\in \mathbb{N}$. The following are
equivalent$:$
\par$(1)$  $y\in {\rm comm}^2(a)$, $yay=y$ and $a^n-ay\in
\sqrt{J(R)}$ $($i.e., $a\in R^{\tiny{\textcircled{\qihao D}}}$  with
$a^{\tiny{\textcircled{\qihao D}}}=y)$.
\par$(2)$ $e^2=e \in  {\rm comm}^2(a)$ and $a^n-e\in
\sqrt{J(R)}$.
\par$(3)$ $y\in {\rm comm}^2(a)$, $yay=y$, $a-a^2y\in
\sqrt{J(R)}$ and $a-a^{n+1}\in \sqrt{J(R)}$.
\par$(4)$ $(1-e)^2=1-e \in  {\rm comm}^2(a)$, $a+(1-e)\in U(R)$, $a(1-e)\in
\sqrt{J(R)}$ and $a-a^{n+1}\in \sqrt{J(R)}$.
\end{corollary}

\begin{example}Let $F$ be a field and $R={\rm M}_k(F)$. Then,
$\sqrt{J(R)}=R^{ \rm nil}$. One infers from \cite[Theorem
2.1]{Wang12} that all matrices in $R$ are $p$-Drazin invertible. But, a simple check shows that
any matrix $A\in R$ satisfying the relation $A-A^{n+1}\notin R^{ \rm nil}$ is not
$pns$-Drazin invertible, as needed.
\end{example}

\begin{corollary}\label{1-5}
Let $a \in R$ and $n\in \mathbb{N}$. The following are equivalent$:$
\par$(1)$ The element $a$ is $pns$-Drazin invertible.
\par$(2)$ There exists $p^2=p\in {\rm comm}^2(a)$ such that $ap\in \sqrt{J(R)}$ and $1-(a+p)^n\in \sqrt{J(R)}.$
\end{corollary}

\begin{proof}
$(1)\Rightarrow(2)$. Let $e=aa^{\tiny{\textcircled{\qihao D}}}$ and
$p=1-e$. In virtue of Corollary \ref{1-4}, one has that $p^2=p\in {\rm comm}^2(a)$,
$a+p\in U(R)$, $ap\in \sqrt{J(R)}$ and $a-a^{n+1}\in \sqrt{J(R)}$.
Write $u=a+p$. By the classical Binomial Theorem, we may write
$$u-u^{n+1}=(a+p)-(a+p)^{n+1}=(a-a^{n+1})-ap\cdot\sum_{i=1}^{n}C_{n+1}^i(ap)^{n-i}\in \sqrt{J(R)}.$$
Since $u\in U(R)$, it follows that $1-u^{n}\in \sqrt{J(R)}$, as
required.

$(2)\Rightarrow(1)$. From $1-(a+p)^n\in \sqrt{J(R)},$ one has
$(a+p)-(a+p)^{n+1}\in \sqrt{J(R)}$, and $a+p\in U(R)$ as $(a+p)^n\in
1+\sqrt{J(R)}\subseteq U(R).$ Utilizing Corollary \ref{1-4}, we only need
to show $a-a^{n+1}\in \sqrt{J(R)}$. Similarly to the above proof, we can get
$$a-a^{n+1}=[(a+p)-(a+p)^{n+1}]+ap\cdot\sum_{i=1}^{n}C_{n+1}^i(ap)^{n-i}\in
\sqrt{J(R)}$$ since $ap\in \sqrt{J(R)}.$ Thus, $(1)$ follows, as promised.
\end{proof}

\medskip


\medskip

\section{\bf $pns$-Drazin inverses with involution}

This section focuses on the study of a class of $pns$-Drazin
inverses in a $*$-ring. Recall once again that a ring $R$ is a \emph{$*$-ring}
(or a \emph{ring with involution}) if there is a map $*:R\rightarrow
R$ such that for all $x,~y\in R,$ $(x+y)^*=x^*+y^*,~(xy)^*=y^*x^*$ and $(x^*)^*=x$.

\begin{definition}
Let $R$ be a $*$-ring and $n\in  \mathbb{N}$. An element $a\in R$ is
called $pns$-$*$-Drazin invertible if there exists $x\in R$ such
that
\begin{center}
 $x\in {\rm comm}^{2}(a)$,~~   ~~$xax=x$,~~  ~~$(ax)^{\ast}=ax$,~~   ~~$a^{n}-ax\in \sqrt{J(R)}.$
\end{center}
In this case, $x$ is called a $pns$-$*$-Drazin inverse of $a$.
\end{definition}

For a $*$-ring $R$ and $n\in \mathbb{N}$. The $pns$-$*$-Drazin
inverse of $a\in R$ will be denoted by $a^\diamond$; and the set of
all $pns$-$*$-Drazin invertible elements of $R$ is denoted by
$R^\diamond$. Apparently, $R^\diamond\subseteq
R^{\tiny{\textcircled{\qihao D}}}$.

\begin{remark}
It is clear that $a\in R^\diamond$ if, and only if, $a\in
R^{\tiny{\textcircled{\qihao D}}}$ and
$(aa^{\tiny{\textcircled{\qihao
D}}})^*=aa^{\tiny{\textcircled{\qihao D}}}$. So, by the uniqueness
of the $pns$-Drazin inverse, one gets that the $pns$-$*$-Drazin
inverse is unique whenever it exists.
\end{remark}

Recall that an element $p$ of a $*$-ring $R$ is a \emph{projection}
(cf. \cite{Ber72}) if $p$ is a self-adjoint idempotent (i.e., $p^2=p$ and
$p=p^*$).

\begin{proposition}\label{1-10}
Let $R$ be a $*$-ring, $a\in R$ and $n\in \mathbb{N}$. The following
are equivalent$:$
\par$(1)$ The element $a$ is $pns$-$*$-Drazin invertible.
\par$(2)$ There exists a projection $p\in {\rm comm}^2{(a)}$ such that $a^n-p\in
\sqrt{J(R)}.$
\end{proposition}

\begin{proof}
$(1)\Rightarrow(2)$. Clearly, $a$ is $pns$-Drazin invertible. Write
$p=aa^{\tiny{\textcircled{\qihao D}}}$. Employing Lemma \ref{1-1},
we have $p^2=p\in {\rm comm}^2{(a)}$ and $a^n-p\in \sqrt{J(R)}$. As
$(aa^{\tiny{\textcircled{\qihao
D}}})^*=aa^{\tiny{\textcircled{\qihao D}}}$, we infer that $p$ is a projection.

$(2)\Rightarrow(1)$. Let $x=(1+a^n-p)^{-1}a^{n-1}p.$ Applying Lemma
\ref{1-1}, we conclude that $x$ is the $pns$-Drazin inverse of $a$. Notice that
$$ax=a^{n}(1+a^n-p)^{-1}p=(a^{n}+1-p)(1+a^n-p)^{-1}p=p.$$ But, since $p$ is
a projection, it follows that $(ax)^*=ax$. So, $x$ is the
$pns$-$*$-Drazin inverse of $a$, as pursued.
\end{proof}

\begin{example}\label{3-3} Let $R$ be a $*$-ring and $n\in \mathbb{N}$.
\par$(1)$ Elements of $\sqrt{J(R)}$ are $pns$-$*$-Drazin invertible. In particular, nilpotents and elements of $J(R)$ are $pns$-$*$-Drazin invertible.
\par$(2)$ Idempotents of $R$ are $pns$-Drazin invertible, and an idempotent $e$ is $pns$-$*$-Drazin invertible if, and only if, $e$ is a projection.
\end{example}

\begin{proof}
$(1)$ As $0$ is a projection, the result follows directly from Proposition
\ref{1-10}.

$(2)$ Idempotents are obviously $pns$-Drazin invertible. Further,
assume that $e$ is $pns$-$*$-Drazin invertible. Then, owing to Proposition
\ref{1-10}, there exists a projection $p\in {\rm comm}^2(e)$ such
that $e^n-p=e-p\in \sqrt{J(R)}$. However, $ep=pe$ ensures
$(e-p)^3=e-p$. Thus, $e=p$. The converse is easy to be obtained.
\end{proof}



On the other vein, Cline \cite{Cline} showed that, for any $a,b$ of a ring $R$, the product $ab$ is Drazin invertible if, and only if, $ba$ is Drazin invertible and $(ba)^D=b((ab)^D)^2a$, where $(ba)^D$ (resp., $(ab)^D$) is the Drazin inverse of $ba$ (resp., $ab$). This equation is known as the Cline's formula. That formula was studied for various aspects of generalized inverses in e.g., \cite{Lian,Liao,Mosic21,Wang12}; see also the bibliography therewith.

\begin{lemma}\label{3-1}
\cite[Corollary3.2]{Mosic21} Let $R$ be a  ring and $a$, $b$, $c\in
R$. If $aba=aca$, then $ac\in R^{\tiny{\textcircled{\qihao D}}}$ if, and only if, $ba\in R^{\tiny{\textcircled{\qihao D}}}$. In this case, $(ac)^{\tiny{\textcircled{\qihao
D}}}=a[(ba)^{\tiny{\textcircled{\qihao D}}}]^{2}c$ and
$(ba)^{\tiny{\textcircled{\qihao
D}}}=b[(ac)^{\tiny{\textcircled{\qihao D}}}]^{2}a$.
\end{lemma}

The next construction reveals that the Cline's formula does {\it not} hold for $pns$-$*$-Drazin inverses.

\begin{example}\label{6-20}
Let $R={\rm M}_2(S)$, where $S$ is a commutative ring. An involution $*$ of $R$ is defined by the transpose of matrices. It is clear that $R$ is a $*$-ring. Let

\begin{center}
$A=\left(\begin{array}{cc}0&0\\1&1\end{array}\right)$,~
$B=\left(\begin{array}{cc}0&1\\0&0\end{array}\right)\in R$.
\end{center}
Then, $AB=\left(\begin{array}{cc}0&0\\0&1\end{array}\right)$ and
$BA=\left(\begin{array}{cc}1&1\\0&0\end{array}\right)$ are both
idempotents. Since $AB$ is a projection of $R$, by Example
\ref{3-3}(2) we deduce that $AB$ is $pns$-$*$-Drazin invertible. However, $BA$ is
not $pns$-$*$-Drazin invertible, because it is not a projection.
\end{example}

We are now prepared to establish the following statement.

\begin{theorem}\label{525-1}
Let $R$ be a $*$-ring and $a$, $b$, $c\in R$. If
$aba=ba^{2}=a^{2}c=aca$, then $ac\in R^\diamond$ if, and only if,
$ba\in R^\diamond$. In this case,
$(ac)^\diamond=a[(ba)^\diamond]^{2}c$ and
$(ba)^\diamond=b[(ac)^\diamond]^{2}a$.
\end{theorem}

\begin{proof} $(\Rightarrow)$. Let $y=b[(ac)^\diamond]^{2}a$. Exploiting Lemma \ref{3-1}, one derives that $y(ba)y=y$, $y\in {\rm comm}^{2}(ba)$, and $(ba)^{n}-(ba)y\in \sqrt{J(R)}$. Since
$a(ba)=(ba)a=a(ac)=(ac)a$, we obtain
\begin{equation*}
\begin{split}
bay=ba  b[(ac)^\diamond]^{2}a
 &=bab\cdot
ac[(ac)^\diamond]^{3}a=b(aca)c[(ac)^\diamond]^{3}a\\
&=b(ac)^\diamond a=ba(ac)^\diamond =ba\cdot
ac[(ac)^\diamond]^{2}\\
&=(aca)c[(ac)^\diamond]^{2}=ac(ac)^\diamond.
\end{split}
\end{equation*}
Therefore, $(bay)^{\ast}=bay$, and thus, $y=ba^\diamond$.

$(\Leftarrow)$. Let $z=a[(ba)^\diamond]^{2}c$. According to Lemma
\ref{3-1}, we have $z(ac)z=z$, $z\in {\rm comm}^{2}(ac)$, and
$(ac)^{n}-(ac)z\in \sqrt{J(R)}$. To show that $ac\in R^\diamond$,
it suffices to prove $(acz)^{\ast}=acz$. As
$a(ba)=(ba)a=a(ac)=(ac)a$, it follows that
\begin{center}
$acz=aba[(ba)^\diamond]^{2}c=a(ba)^\diamond c=(ba)^\diamond ac
=[(ba)^\diamond]^{2}ba\cdot
ac=[(ba)^\diamond]^{2}b(aba)=(ba)^\diamond ba$.
\end{center}
So, $(acz)^{\ast}=acz$, as required.
\end{proof}

Putting ``$b=c$" in Theorem \ref{525-1}, we yield the following result immediately.

\begin{corollary}
Let $R$ be a $*$-ring and $a$, $b\in R$. If $aba=ba^{2}=a^{2}b$, then $ab\in R^\diamond$ if, and only if, $ba\in R^\diamond$.
\end{corollary}

Let $R$ be a ring and $a,b,c\in R$. The standard Jacobson's lemma states that if $1-ab\in U(R)$, then $1-ba\in U(R)$ and $(1-ba)^{-1}=1+b(1-ab)^{-1}a$. However, Corach et al. \cite{Corach} extended the classical Jacobson's lemma to the case ``$aba=aca$", and proved that if $1-ac\in U(R)$, then $1-ba\in U(R)$.

\medskip

We now have the following statement for $pns$-Drazin inverses which somewhat improves the corresponding results from \cite{LN} and \cite{RL}.

\begin{lemma}\label{3-2}
Let $R$ be a ring, $a$, $b$, $c\in R$ and $n\in \mathbb{N}$. If
$aba=aca$, then $\alpha=1-ba$ is $pns$-Drazin inverses if, and only if, so is $\beta=1-ac$. In this case,
\begin{center}
$\alpha^{\tiny{\textcircled{\qihao
D}}}=1-b\beta^{\pi}v(ac)^{2}a+b\beta\beta^{\tiny{\textcircled{\qihao
D}}}a+b\beta^{\tiny{\textcircled{\qihao D}}}aba$,
\end{center}
and
\begin{center}
$\beta^{\tiny{\textcircled{\qihao
D}}}=1-a\alpha^{\pi}u(ba)^{2}c+a\alpha\alpha^{\tiny{\textcircled{\qihao
D}}}c+a\alpha^{\tiny{\textcircled{\qihao D}}}cac$,
\end{center}
where $\alpha^{\pi}=1-\alpha\alpha^{\tiny{\textcircled{\qihao D}}}$,
$\beta^{\pi}=1-\beta\beta^{\tiny{\textcircled{\qihao D}}}$,
$v=(1-\beta\beta^{\pi}\sum_{i=0}^{2}(ac)^{i})^{-1}$ and
$u=(1-\alpha\alpha^{\pi}\sum_{i=0}^{2}(ba)^{i})^{-1}$.
\end{lemma}

\begin{proof}
It is enough to prove only one direction, as the other direction can be
showed in a similar manner. So, we suppose that $\alpha=1-ba\in
R^{\tiny{\textcircled{\qihao D}}}$. Let
$\alpha^{\pi}=1-\alpha\alpha^{\tiny{\textcircled{\qihao D}}}$. In
accordance with Corollary \ref{1-4}, $\alpha\alpha^{\pi}\in \sqrt{J(R)}$.
So, $t=1-\alpha\alpha^{\pi}\sum_{i=0}^{2}(ba)^{i}\in U(R)$. Let
$u=t^{-1}$ and $x=1-a\alpha^{\pi}u
(ba)^{2}c+a\alpha\alpha^{\tiny{\textcircled{\qihao
D}}}c+a\alpha^{\tiny{\textcircled{\qihao D}}}cac$. Note that
$(ba)u=u(ba)$. To show that $\beta=1-ac\in
R^{\tiny{\textcircled{\qihao D}}}$, we next prove the following three things: (1)~$x\in
{\rm comm}^{2}(\beta)$;~ ~(2)~$x\beta x=x$;~
~(3)~$\beta^{n}-x\beta\in \sqrt{J(R)}$.

(1) Let $y\in R$ with $y\beta=\beta y$. Then, $ac$ commutes with $y$.
So, we obtain
\begin{equation}
(bacya)\alpha =bacya-bacacya =\alpha(bacya ),\tag{1.1}
\end{equation}
which implies that $bacya$ commutes with
$\alpha^{\tiny{\textcircled{\qihao D}}}$, $\alpha^{\pi}$, $ba$ and
$u$. Notice that $\alpha^{\pi}u^{-1}
=\alpha^{\pi}-\alpha^{\pi}(1-(ba)^{3}) =\alpha^{\pi}(ba)^{3}$, and hence we
have
\begin{equation}
\alpha^{\pi}=\alpha^{\pi}(ba)^{3}u.\tag{1.2}
\end{equation}
By (1.1), we obtain
\begin{equation}
ya\alpha^{\pi}u(ba)^{2}c
=ya(ba)^2\alpha^{\pi}uc=a(bacya)\alpha^{\pi}uc
=a\alpha^{\pi}u(bacya)c =a\alpha^{\pi}u(ba)^{2}cy.\tag{1.3}
\end{equation}
Multiplying (1.3) by $ac$ on the right, we get
$ya\alpha^{\pi}c=a\alpha^{\pi}cy$ with the aid of (1.2), which yields
\begin{equation}
ya\alpha\alpha^{\tiny{\textcircled{\qihao
D}}}c=a\alpha\alpha^{\tiny{\textcircled{\qihao D}}}cy.\tag{1.4}
\end{equation}
It now follows that
\begin{equation}
yaba\alpha\alpha^{\tiny{\textcircled{\qihao
D}}}cac=aca\alpha\alpha^{\tiny{\textcircled{\qihao
D}}}cyac=aba\alpha\alpha^{\tiny{\textcircled{\qihao D}}}cacy
\tag{1.5}
\end{equation}
by multiplying $ac$ on both sides of (1.4). Besides, in view of (1.1),
$$y(a(ba)^{2}\alpha^{\tiny{\textcircled{\qihao
D}}}cac)=ya(ba)^{3}\alpha^{\tiny{\textcircled{\qihao
D}}}c=a(bacya)ba\alpha^{\tiny{\textcircled{\qihao D}}}c
=aba\alpha^{\tiny{\textcircled{\qihao
D}}}(bacya)c=(a(ba)^{2}\alpha^{\tiny{\textcircled{\qihao
D}}}cac)y,$$ which gives
\begin{equation}y(aba(1-\alpha)\alpha^{\tiny{\textcircled{\qihao
D}}}cac)=(aba(1-\alpha)\alpha^{\tiny{\textcircled{\qihao
D}}}cac)y.\tag{1.6}
\end{equation}
Combining (1.5) with (1.6), we obtain
\begin{equation}
yaba\alpha^{\tiny{\textcircled{\qihao
D}}}cac=aba\alpha^{\tiny{\textcircled{\qihao D}}}cacy.\tag{1.7}
\end{equation}
Now, multiplying (1.4) by $ac$ on the right, we have
$ya\alpha\alpha^{\tiny{\textcircled{\qihao
D}}}cac=a\alpha\alpha^{\tiny{\textcircled{\qihao D}}}cacy$, and then
$ya(1-ba)\alpha^{\tiny{\textcircled{\qihao
D}}}cac=a(1-ba)\alpha^{\tiny{\textcircled{\qihao D}}}cacy$. By
virtue of (1.7), we have
\begin{equation}
ya\alpha^{\tiny{\textcircled{\qihao
D}}}cac=a\alpha^{\tiny{\textcircled{\qihao D}}}cacy.\tag{1.8}
\end{equation}
According to (1.3) (1.4) and (1.8), we get $yx=xy$, which implies that
$x\in {\rm comm}^{2}(\beta)$.

(2) Taking $y=ac$ in (1.8), one has that $(a\alpha^{\tiny{\textcircled{\qihao D}}}cac)ac=ac(a\alpha^{\tiny{\textcircled{\qihao
 D}}}cac)$. Then, we can calculate that
\begin{align*}
x\beta=&1-ac-a\alpha^{\pi}u(ba)^{2}c+a\alpha^{\pi}c+a\alpha\alpha^{\tiny{\textcircled{\qihao
D}}}c-a\alpha\alpha^{\tiny{\textcircled{\qihao D}}}cac
+a\alpha^{\tiny{\textcircled{\qihao
D}}}cac-a\alpha^{\tiny{\textcircled{\qihao D}}}cacac\\
=&
1-a\alpha^{\pi}u(ba)^{2}c+(-ac+a\alpha^{\pi}c+a(1-a^{\pi})c)+(-a\alpha\alpha^{\tiny{\textcircled{\qihao
D}}}cac +a\alpha^{\tiny{\textcircled{\qihao
D}}}cac)-a\alpha^{\tiny{\textcircled{\qihao D}}}cacac\\
=&1-a\alpha^{\pi}u(ba)^{2}c+aba\alpha^{\tiny{\textcircled{\qihao
D}}}cac -aca\alpha^{\tiny{\textcircled{\qihao D}}}cac\\
=&1-a\alpha^{\pi}u(ba)^{2}c.
\end{align*}
We, however, see that $(a\alpha^{\pi}u(ba)^{2}c)^2=a\alpha^{\pi}u(ba)^{2}c$ and
$\alpha^{\pi}\alpha^{\tiny{\textcircled{\qihao D}}}=0.$ So,
$$x\beta x=(1-a\alpha^{\pi}u(ba)^{2}c)(1-a\alpha^{\pi}u(ba)^{2}c+a\alpha\alpha^{\tiny{\textcircled{\qihao
D}}}c+a\alpha^{\tiny{\textcircled{\qihao D}}}cac)=x.$$

(3) Observe that
\begin{center}
 $\alpha^{n}-\alpha\alpha^{\tiny{\textcircled{\qihao D}}}
=(1-ba)^{n}-1+1-\alpha\alpha^{\tiny{\textcircled{\qihao D}}}
=\sum_{i=1}^{n}(-1)^{i}C^{i}_{n}(ba)^{i}+\alpha^{\pi} $.
\end{center}
Since
$x\beta=1-a\alpha^{\pi}u(ba)^{2}c$, we obtain
\begin{center}
$\beta^{n}-x\beta =(1-ac)^{n}-1+a\alpha^{\pi}u(ba)^{2}c
=a(\alpha^{\pi}u(ba)^{2}c+\sum_{i=1}^{n}(-1)^{i}C^{i}_{n}(ca)^{i-1}c).$
\end{center}
Hence, one computes that
\begin{align*}
&(\beta^{n}-x\beta)^{2}\\
=&a[\alpha^{\pi}+\sum_{i=1}^{n}(-1)^{i}C^{i}_{n}(ca)^{i}][\alpha^{\pi}u(ba)^{2}c+\sum_{i=1}^{n}(-1)^{i}C^{i}_{n}(ca)^{i-1}c]\\
=&a[\alpha^{\pi}u(ba)^{2}c+\alpha^{\pi}\sum_{i=1}^{n}(-1)^{i}C^{i}_{n}(ca)^{i-1}c
+\sum_{i=1}^{n}(-1)^{i}C^{i}_{n}(ba)^{i}\alpha^{\pi}u(ba)^{2}c\\ &+\sum_{i=1}^{n}(-1)^{i}C^{i}_{n}(ba)^{i}\sum_{i=1}^{n}(-1)^{i}C^{i}_{n}(ca)^{i-1}c]\\
=&a[\alpha^{\pi}+\sum_{i=1}^{n}(-1)^{i}C^{i}_{n}(ba)^{i}][\alpha^{\pi}u(ba)^{2}c+\sum_{i=1}^{n}(-1)^{i}C^{i}_{n}(ca)^{i-1}c]\\
=&a(\alpha^{n}-\alpha\alpha^{\tiny{\textcircled{\qihao
D}}})[\alpha^{\pi}u(ba)^{2}c+\sum_{i=1}^{n}(-1)^{i}C^{i}_{n}(ca)^{i-1}c].
\end{align*}
Repeating the same procedure as that alluded to above, we have
\begin{center}
$(\beta^{n}-x\beta)^{m+1}=a(\alpha^{n}-\alpha\alpha^{\tiny{\textcircled{\qihao
D}}})^m(\alpha^{\pi}u(ba)^{2}c+\sum_{i=1}^{n}(-1)^{i}C^{i}_{n}(ca)^{i-1}c)$
\end{center}
for any integer $m>1.$ As
$\alpha^{n}-\alpha\alpha^{\tiny{\textcircled{\qihao D}}}\in
\sqrt{J(R)},$ one must get $(\beta^{n}-x\beta)^{k}\in J(R)$ for some
integer $k$, which implies that $\beta^{n}-x\beta \in\sqrt{J(R)}$, as required.
\end{proof}

The next comments are worthy of mentioning.

\begin{remark}
\par$(1)$ Jacobson's lemma does not hold for $pns$-$*$-Drazin inverses. To show that, let $R$ be the $*$-ring as given in Example \ref{6-20}, and let
$A=\left(\begin{array}{cc}0&0\\1&1\end{array}\right)$,
$B=\left(\begin{array}{cc}0&1\\0&0\end{array}\right)\in R$. Then,
$1-AB=\left(\begin{array}{cc}1&0\\0&0\end{array}\right)$ and
$1-BA=\left(\begin{array}{cc}0&-1\\0&1\end{array}\right)$. Clearly,
$1-AB$ is a projection, and $1-BA$ is an idempotent but not a
projection. Consulting with Example \ref{3-3}(2), we readily see that $1-AB$ is $pns$-$*$-Drazin
inverses, while $1-BA$ is not.
\par$(2)$ An element $a$ of a $*$-ring is $pns$-$*$-Drazin
invertible cannot imply that so is $1-a$. For example, let
$R=\mathbb{Z}_6$ be the ring of integers modulo $6$. The involution
of $R$ is defined by $*=1_R$, where $1_R$ is the identity
endomorphism of $R$. It is clear that $\sqrt{J(R)}=0$ and $0, 1$ are
the only projections of $R$. For $n=2\in \mathbb{N}$, due to
Proposition \ref{1-10}, $a=-1\in R$ is $pns$-$*$-Drazin invertible,
whereas $1-a=2\in R$ is manifestly not.
\end{remark}

We are now ready to prove the following assertion.

\begin{theorem}\label{525-2}
Let $R$ be a $*$-ring and $a$, $b$, $c\in R$. If
$aba=ba^{2}=a^{2}c=aca$, then $\alpha=1-ba\in R^\diamond$ if, and
only if, $\beta=1-ac\in R^\diamond$. In this case,
\begin{center}
$\alpha^\diamond=1-b\beta^{\pi}v(ac)^{2}a+b\beta\beta^\diamond
a+b\beta^\diamond aba$,
\end{center}
and
\begin{center}
$\beta^\diamond=1-a\alpha^{\pi}u(ba)^{2}c+a\alpha\alpha^\diamond
c+a\alpha^\diamond cac$,
\end{center}
where $\alpha^{\pi}=1-\alpha\alpha^\diamond$, $\beta^{\pi}=1-\beta\beta^\diamond$,
$v=(1-\beta\beta^{\pi}\sum_{i=0}^{2}(ac)^{i})^{-1}$ and $u=(1-\alpha\alpha^{\pi}\sum_{i=0}^{2}(ba)^{i})^{-1}$.
\end{theorem}

\begin{proof} Assume that $\alpha=1-ba$ is $pns$-$*$-Drazin invertible. Let $$u=(1-\alpha\alpha^{\pi}\sum_{i=0}^{2}(ba)^{i})^{-1}$$ and
$$x=1-a\alpha^{\pi}u(ba)^{2}c+a\alpha\alpha^\diamond
c+a\alpha^\diamond cac.$$ Consulting with Lemma \ref{3-2}, one concludes that $x\beta x=x$, $x\in
{\rm comm}^{2}(\beta)$ and $\beta^{n}-\beta x\in \sqrt{J(R)}$. To
show that $\beta \in R^\diamond$, it suffices to prove $(\beta
x)^{\ast}=\beta x$. However, in view of the proof of Lemma \ref{3-2}, it must be that
$\alpha^{\pi}=\alpha^{\pi}(ba)^{3}u$ and $\beta
x=1-a\alpha^{\pi}u(ba)^{2}c$. Since $a\in {\rm comm}(ba),$ the element $a$
commutes with both $\alpha^{\pi}u$ and $ba$. By hypothesis, we obtain
\begin{center}
$\beta x
=1-a\alpha^{\pi}u(ba)^{2}c=1-\alpha^{\pi}u(ba)^{2}ac=1-\alpha^{\pi}u(ba)^{3}=1-\alpha^{\pi}
=\alpha\alpha^\diamond.$
\end{center}
Thus,
$(\beta x)^{\ast}=\beta x$, as desired. The converse can be shown analogously.
\end{proof}

As an automatic consequence, we yield:

\begin{corollary}
Let $R$ be a $*$-ring and $a$, $b\in R$. If $aba=ba^{2}=a^{2}b$, then $\alpha=1-ba\in R^\diamond$ if, and only if, $\beta=1-ab\in R^\diamond$.
\end{corollary}

\medskip


\medskip

\section{\bf $pns$-Drazin inverses for all $n\geq 1$}

Here are focussing on the description of some crucial properties of rings for which every element has a $pns$-Drazin inverse (the integer $n$ depending on the choice of the element). We shall say that a ring $R$ is \textit{pseudo $\pi$-polar}, provided all of its elements are $pns$-Drazin invertible for various naturals $n$; or, equivalently, for each $a\in R$, there exists $e^2=e\in {\rm comm}^2(a)$ such that $a^n-e\in \sqrt{J(R)}$ for some integer $n\geq
1$. Thus, this leads to the interesting fact that pseudo $\pi$-polar rings are strongly clean in the sense that any of their elements is a sum of a unit and an idempotent which commute.

On the other side, it was considered in \cite{CD} those rings $R$ in which every $u\in U(R)$ satisfies the relation $u^n-1\in R^{\rm nil}$ for some $n\geq 1$ (which are called \emph{$\pi$-UU rings}). It was shown there that strongly $\pi$-regular $\pi$-UU rings are always periodic in the sense that, for each $a\in R$, there are two integers $m>n\geq 1$ with $a^m = a^n$. It is long known that $J(R)$ is nil whenever $R$ is strongly $\pi$-regular. However, this implication is still unproved for $\pi$-UU rings.

And so, our first goal here is to find a suitable connection between strongly $\pi$-regular rings, $\pi$-UU rings and pseudo $\pi$-polar rings. That is, as aforementioned, it is principally known that both strongly $\pi$-regular and pseudo $\pi$-polar rings are strongly clean. In this vein, with the results established above at hand, it is not too hard to prove that pseudo $\pi$-polar rings $R$ with nil $J(R)$ are always periodic, and vice versa.

\medskip

Concretely, the following is fulfilled.

\begin{theorem}\label{2-2}
A ring $R$ is periodic if, and only if, $R$ is pseudo $\pi$-polar and $J(R)$ is nil.
\end{theorem}

\begin{proof}
Assume that $R$ is periodic. Then, $R$ is strongly $\pi$-regular. In view of \cite[Theorem 2.1]{Wang12}, $R$ is pseudo-polar and $J(R)$ is nil. Furthermore, let $u\in U(R)$. Then, there exist integers $k>l\geq1$ such that
$u^{k}=u^{l}$. So, $1-u^{k-l}=0\in \sqrt{J(R)}$. Hence, $R$ is pseudo-$\pi$-polar by exploiting Theorem \ref{1-3}, as claimed.

Conversely, letting $a\in R$, in view of Theorem \ref{1-3} we write $a-a^{n+1}\in \sqrt{J(R)}$ for some integer $n\geq 1$. Since $J(R)$ is nil, we get $\sqrt{J(R)}=R^{\rm nil}$. Thus, $a-a^{n+1}$ in nilpotent, i.e., $$(a-a^{n+1})^l=a^{l}(1-a^n)^l=0$$ for some positive integer $l$. This yields that $a^l=a^{l+1}f(a)$ for some polynomial $f(x)\in \mathbb{Z}[x]$. Finally, by virtue of \cite[Theorem 2.1]{bell}, we conclude that $R$ is a periodic ring, as asserted.
\end{proof}

Recall that an element $a$ of a ring $R$ is said to be \emph{quasi-nilpotent} \cite{Harte} if $1-ax\in U(R)$ for all $x\in {\rm comm}(a)$. The set of all quasi-nipotent elements of $R$ is denoted by $R^{\rm qnil}$. We can infer from \cite[Example 4.2]{Wang12} that $\sqrt{J(R)}$ is properly contained in $R^{\rm qnil}$.

Nevertheless, the following is true.


\begin{proposition}
If $R$ is a pseudo $\pi$-polar ring, then $\sqrt{J(R)}=R^{\rm qnil}$.
\end{proposition}

\begin{proof}
It suffices to show that $R^{\rm qnil} \subseteq \sqrt{J(R)}$. To that purpose, let $q\in R^{\rm qnil}$. Then, there is an integer $n\geq 1$ such that $q-q^{n+1}\in\sqrt{J(R)}$. Write $u=1-q^n$. Since $q\in R^{\rm qnil}$, we get $u\in U(R)$ and $uq=qu\in \sqrt{J(R)}$. This implies $q=(qu)u^{-1}\in \sqrt{J(R)}$. Therefore, $\sqrt{J(R)}=R^{\rm
qnil}$, as required.
\end{proof}


\medskip

Thereby, combining the results quoted above, one verifies that the following relationships are valid:

\medskip

\centerline{strongly $\pi$-regular + pseudo $\pi$-polar $\iff$ periodic $\iff$ strongly $\pi$-regular + $\pi$-UU.}

\medskip

Taking into account the discussion alluded to above, a question which automatically arises is what can be said about pseudo $\pi$-polar $\pi$-UU rings, i.e., are they strongly $\pi$-regular and thus periodic? We note that there is a pseudo $\pi$-polar ring which is {\it not} $\pi$-UU. Indeed, let $R=\mathbb{Z}_{(p)}$ be the localization of the ring of integers $\mathbb{Z}$ at the prime ideal $(p)$, where $p$ is a prime integer. In virtue of \cite[Example
2.10]{Wang12}, $R$ is pseudo-polar. Observe that $R/J(R)\cong \mathbb{Z}_p$. Furthermore, for each $u\in U(\mathbb{Z}_{(p)})$, one has that $\overline{u}^{p-1}=\overline{1}\in R/J(R)$. Thus, $1-u^{p-1}\in J(R)$. Utilizing Theorem \ref{1-3}, $R$ is a pseudo-$\pi$-polar ring, but manifestly it is {\it not} $\pi$-UU.

Specifically, we pose the following which, hopefully, makes some sense:

\begin{conjecture} Pseudo $\pi$-polar $\pi$-UU rings are strongly $\pi$-regular, and hence periodic.
\end{conjecture}

We just notice that, if  $J(R)$ is nil, this will imply a positive resolution of above problem.

\medskip

We continue our considerations with some matrix results like these in the sequel. We first need the following technicality.

\begin{proposition}\label{corner}
If $R$ is a pseudo-$\pi$-polar ring, then so is the corner subring $eRe$ for any $e^2=e\in R$.
\end{proposition}

\begin{proof}
Let $a\in eRe$. One inspects that Theorem \ref{1-3} allows us to derive that $a-a^{n+1}\in \sqrt{J(R)}$
for some integer $n\geq 1$. So, there exists an integer $k$ such that $(a-a^{n+1})^k\in J(R)$. Observe also that $(a-a^{n+1})^k\in eRe$. Then, $$(a-a^{n+1})^k\in J(R)\cap eRe=J(eRe),$$ that is, $a-a^{n+1} \in\sqrt{J(eRe)}$. Furthermore, as $R$ is pseudo-polar, in accordance with \cite[Corollary 2.10]{Wang12} we conclude that $eRe$ is pseudo-polar too. Thus, the result follows now immediately.
\end{proof}

All local rings are known to be pseudo-polar (see cf. \cite{Wang12}). But, interestingly, there exist local rings which are {\it not} pseudo $\pi$-polar, e.g., such as the rational number field $\mathbb{Q}$ and any field that contains $\mathbb{Q}$. It is easy to see that a local ring $R$ is pseudo $\pi$-polar if, and only if, for each $u\in
U(R)$, $u^n\in 1+ J(R)$ for some $n\geq 1$. For a convenience in the exposition, we call a local ring $R$ \emph{special} if, for every $u\in U(R)$, $u^n\in 1+ J(R)$ for some $n\geq 1$. For example, the considered above $\mathbb{Z}_{(p)}$ is a special local ring.

\medskip

For a ring $R$, the $n\times n$ upper triangular matrix ring over $R$ is denoted by ${\rm T}_n(R)$.

\begin{proposition}
Let $R$ be a commutative local ring. The following are equivalent$:$
\par$(1)$ $R$ is a special local ring.
\par$(2)$ ${\rm T}_n(R)$ is pseudo $\pi$-polar for any integer $n\geq 1$.
\par$(3)$ ${\rm T}_n(R)$ is pseudo $\pi$-polar for some integer $n\geq 1$.
\end{proposition}

\begin{proof}
$(1)\Rightarrow(2)$. Since $R$ is a commutative local ring, by using \cite[Theorem 2.13]{Wang12} ${\rm T}_n(R)$ is pseudo-polar. One verifies that, in view of Theorem \ref{1-3}, we only need to show that, for every $V=(v_{ij})\in U({\rm T}_n(R))$, there exists integer $k\geq 1$ such that $I-U^k\in J({\rm T}_n(R))$. In fact, $V\in U({\rm T}_n(R))$ if, and only if, $v_{ii} \in U(R)$ for each $i$. As $R$ is a special local ring, we check that $1-v_{ii}^{k_i}\in J(R)$, where $k_i\geq 1$ and $i=1,\ldots,n$. Put $k=k_1k_2\cdots k_n$. Thus,  $$1-v_{ii}^k=(1-v_{ii}^{k_i})c \in
J(R)$$ for some $c\in R$. It follows immediately that $I-U^k\in J({\rm T}_n(R))$, as needed. So, ${\rm T}_n(R)$ is pseudo $\pi$-polar, as promised.

$(2)\Rightarrow(3)$ is clear.

$(3)\Rightarrow(1)$. Taking into account Proposition \ref{corner}, $R$ is pseudo $\pi$-polar, and the rest is quite obvious.
\end{proof}

It can be inferred from \cite{Chen06} that matrix rings over a special local ring need {\it not} be strongly clean. So, in particular, matrix rings over pseudo $\pi$-polar rings are {\it not} always pseudo $\pi$-polar.

It was established in \cite{CD} that if $R$ is a ring such that $a^m=a$ for all $a\in R$ and a fixed integer $m> 1$, then, for any positive integer $n$, the matrix ring ${\rm M}_n(R)$ is periodic. We now expand this to the following affirmation.

\begin{proposition}
Let $R$ be a commutative ring. If $a-a^m\in R^{\rm nil}$ for each $a\in R$ and a fixed integer $m>1$, then ${\rm M}_n(R)$ is pseudo $\pi$-polar for any $n\geq 1$.
\end{proposition}

\begin{proof}
In view of \cite{KosanZ}, $R$ is strongly $\pi$-regular and every element $a$ from $\overline{R}=R/J(R)$ satisfies the identity $\overline{a}^m=\overline{a}$. Since $R$ is a commutative strongly $\pi$-regular ring, it follows from \cite{Bo} that ${\rm M}_n(R)$ is strongly $\pi$-regular. Bearing in mind \cite[Theorem 2.1]{Wang12}, one concludes that ${\rm M}_n(R)$ is pseudo-polar.

Now, we finish the proof by showing that, for every $V\in U(M_n(R))$, the relation $1-V^k \in J({\rm M}_n(R))$ holds for some integer $k$ or, in an equivalent manner, that, for each element
$$\overline{V}\in U(M_n(\overline{R}))=U(M_n(R)/J(M_n(R))),$$ the equality $\overline{V}^k=\overline{1}$ holds for some integer $k$. So, with no harm in generality, we may assume that $J(R)=0$. Hence, $a=a^m$ for all $a\in R$. By the arguments of \cite[p. 197]{Lam}, $R$ is a subdirect product of its left primitive homomorphic images $R_i$ such that each $R_i$ is a field satisfying $a_i^m=a_i$ for all $a_i\in R_i$ and $i\in \Lambda$. Therefore, $|R_i|\leq m$ for all $R_i$, and thus it must be that $|{\rm M}_n(R_i)|\leq m^{n^2}$. Set $k=m^{n^2}!$. Then, for any $U_i\in U({\rm M}_n(R_i))$, we can get that $U_i^{k}=1$. Observing that ${\rm M}_n(R)$ can be viewed as a unitary subring of ${\rm M}_n(\prod_{i\in \Lambda} R_i)\cong \prod_{i\in \Lambda} {\rm M}_n(R_i)$. Let $V=(U_i)_{i\in \Lambda}$ be a unit of $\prod_{i\in \Lambda} {\rm M}_n(R_i)$. Consequently, for each $i$, $U_i$ is a unit of ${\rm M}_n(R_i)$. From the equality $U_i^{k}=1$, we obtain that $V^k=1$, as required.
\end{proof}

\medskip
\medskip

\centerline {\bf ACKNOWLEDGMENTS:} The research of the first-named and third-named authors was supported by the Key Laboratory of Financial Mathematics of Fujian Province University (Putian University) (No. JR202203), Anhui Provincial Natural Science Foundation (No. 2008085MA06) and the Key project of Anhui Education Committee (No. gxyqZD2019009). The research of the second-named author was supported by the projects Junta de Andaluc\'ia under Grant (FQM 264) and T\"UB\'ITAK under Grant (BIDEB 2221).

\end{document}